\documentclass[12pt,a4paper]{article}

\usepackage{amsmath,amssymb,amsfonts,amsthm,bbold}

\newtheorem{proposition}{Proposition} 
\newtheorem{corollary}{Corollary}
\usepackage[margin=2cm]{geometry}
\usepackage[round,sort,comma]{natbib}
\usepackage{graphicx} 
\usepackage{xcolor}
\usepackage[colorlinks=true,linkcolor=red,citecolor=red,urlcolor=blue]{hyperref}
\usepackage[title]{appendix}
\usepackage{setspace}
\usepackage{booktabs}
\title{\textbf{Multiscale Asymptotic Normality in Quantile Regression: Hilbert Matrices and Polynomial Designs}}
\usepackage{authblk}

\author[1]{Sa\"{i}d Maanan\thanks{Corresponding author: \texttt{maanan.said@gmail.com}}}
\author[2]{Azzouz Dermoune}
\author[1]{Ahmed El Ghini}

\affil[1]{LEAM, Mohammed V University in Rabat,  Morocco}
\affil[2]{Université de Lille, Laboratoire Paul Painlevé, Lille, France}

\date{}
\newcommand{\keywords}[1]{\par\noindent\textbf{Keywords:} #1}

\begin{document}
\maketitle
\begin{abstract}
This paper investigates the asymptotic properties of quantile regression estimators in linear models, with a particular focus on polynomial regressors and robustness to heavy-tailed noise. Under independent and identically distributed (i.i.d.) errors with continuous density around the quantile of interest, we establish a general Central Limit Theorem (CLT) for the quantile regression estimator under normalization using \(\Delta_n^{-1}\), yielding asymptotic normality with variance \( \tau(1-\tau)/f^2(0) \cdot D_0^{-1} \). In the specific case of polynomial regressors, we show that the design structure induces a Hilbert matrix in the asymptotic covariance, and we derive explicit scaling rates for each coefficient. This generalizes Pollard’s and Koenker’s earlier results on LAD regression to arbitrary quantile levels \( \tau \in (0, 1) \). We also examine the convergence behavior of the estimators and propose a relaxation of the standard CLT-based confidence intervals, motivated by a theoretical inclusion principle. This relaxation replaces the usual \( T^{j+1/2} \) scaling with \( T^\alpha \), for \( \alpha < j + 1/2 \), to improve finite-sample coverage. Through extensive simulations under Laplace, Gaussian, and Cauchy noise, we validate this approach and highlight the improved robustness and empirical accuracy of relaxed confidence intervals. This study provides both a unifying theoretical framework and practical inference tools for quantile regression under structured regressors and heavy-tailed disturbances.
\end{abstract}


\section{Introduction}  
\label{section_one}

Quantile regression is a cornerstone of robust regression analysis, widely employed in time series modeling and other statistical applications due to its resilience against outliers and heavy-tailed noise \citep{Fox2018, Koenker2005}. Unlike classical least squares (LS) estimation, which minimizes the sum of squared residuals, quantile regression minimizes an asymmetrically weighted sum of absolute residuals, providing greater robustness in scenarios where the noise distribution exhibits heterogeneity or heavy tails.

To illustrate this robustness, consider the simple model \(y_i = \beta^* + u_i,\ i = 1, \ldots, n\), where \(u_i\) are i.i.d.\ and follow a Cauchy distribution. In this case, the median of the observations \(y_i\) is the least absolute deviation (LAD) estimator \(\hat{\beta}_{1/2}\) of \(\beta^*\), which satisfies asymptotic normality \citep{Stigler1973}, while the least squares estimator \(\hat{\beta}_{\text{LS}}\), equal to the empirical mean of the \(y_i\), neither converges in probability nor satisfies asymptotic normality. This highlights the limitations of least squares under heavy-tailed noise and the broader utility of quantile regression.

\citet{Bassett1978} considered the problem of estimating the \(p\)-dimensional parameter vector \(\beta^*\) in the linear model
\[
y_i = x_i^\top \beta^* + u_i, \quad i=1, \ldots, n,
\]
where \(x_i^\top = (x_{i1}, \ldots, x_{ip})\) is a row vector of regressors. Assuming that the errors \(u_i\) are i.i.d.\ with a cumulative distribution function \(F\) having median zero is equivalent to saying that the conditional median of \(y_i\) given \(x_i\) equals \(x_i^\top \beta^*\). The LAD estimator of \(\beta^*\), also denoted by \(\hat{\beta}_{1/2}\), minimizes the objective function:
\[
g(\beta) = \sum_{i=1}^n |y_i - x_i^\top \beta|.
\]

When the error density \(f\) is positive and continuous at zero and the design satisfies \(\frac{1}{n} \sum_{i=1}^n x_i x_i^\top \to P\) with \(P\) positive definite, \citet{Bassett1978} showed that
\[
\sqrt{n}(\hat{\beta}_{1/2} - \beta^*) \xrightarrow{d} N(0, \omega^2 P^{-1}),
\]
where \(\omega^2 = \frac{1}{4f^2(0)}\) is the asymptotic variance of the sample median.

Later, \citet{Pollard1991} provided a more geometric proof using convexity tools and showed that under (A1) \( \sum_{i=1}^n x_i x_i^\top \) is positive definite, and (A2) \( \max_{i \leq n} \| \left( \sum_{i=1}^n x_i x_i^\top \right)^{-1/2} x_i \| \to 0 \) as \( n \to \infty \), we have
\[
2f(0) \left( \sum_{i=1}^n x_i x_i^\top \right)^{1/2} (\hat{\beta}_{1/2} - \beta^*) \xrightarrow{d} N(0, I_p),
\]
where \(I_p\) is the identity matrix of size \(p \times p\). Observe that if \(\frac{1}{n} \sum_{i=1}^n x_i x_i^\top \to P\), then
\[
\sqrt{n}(\hat{\beta}_{1/2} - \beta^*) \xrightarrow{d} N\left(0,\, \frac{1}{4f^2(0)} P^{-1}\right),
\]
recovering the result of \citet{Bassett1978}.

A generalization to arbitrary quantiles \(\tau \in (0,1)\) is given by \citet[Theorem 4.1]{Koenker2005}. The quantile regression estimator \(\hat{\beta}_\tau\) is defined as the minimizer of the check loss function:
\[
g_\tau(\beta) = \sum_{i=1}^n \rho_\tau(y_i - x_i^\top \beta), \quad \text{where} \quad \rho_\tau(u) = u(\tau - \mathbb{I}_{\{u < 0\}}).
\]
Note that \(g_{1/2}(\beta) = \frac{1}{2} \sum |y_i - x_i^\top \beta|\), recovering the LAD case. Under (A0) \(P(u_i < 0) = \tau\), and (A1) \( \frac{1}{n} \sum_{i=1}^n x_i x_i^\top \to D_0 \), with \(D_0\) positive definite, Koenker proves:
\[
\sqrt{n}(\hat{\beta}_\tau - \beta^*) \xrightarrow{d} N\left(0,\, \frac{\tau(1-\tau)}{f^2(0)} D_0^{-1} \right).
\]
This result provides the foundational framework for analyzing quantile regression estimators in a broad range of settings.

Despite the extensive body of research on quantile regression, its application to polynomial regressors has received relatively little attention. Polynomial regression is particularly useful for capturing non-linear trends and deterministic components in data. Extending quantile regression to this setting presents unique challenges, especially in characterizing the asymptotic behavior of the estimators and understanding their scaling properties under non-standard design matrices.

In this paper, we address these challenges by establishing the multiscale asymptotic normality of quantile regression estimators for polynomial models. A key contribution is the derivation of the asymptotic precision matrix, which we show to be proportional to Hilbert matrices, a class of structured matrices with well-understood mathematical properties. Our results specialize Koenker's general asymptotic theory to polynomial regressors, and refine it by explicitly characterizing the covariance structure in terms of Hilbert matrices. 

We also analyze the speed of convergence of quantile regression estimators to the true parameters, both in probability and almost surely, through a combination of theoretical results and simulation studies. These simulations provide insights into the robustness and efficiency of quantile regression under various noise distributions, including Laplace, Gaussian, and Cauchy.

This work bridges theoretical advancements with practical implications, showcasing the versatility of quantile regression in capturing complex trends and ensuring robust performance in challenging data environments. Three empirical applications—spanning environmental, climatic, and energy domains—demonstrate the method’s practical utility: (i) modeling drought risk via minimal river discharge (\(\tau = 0.05\)) in Brazil’s S\~{a}o Francisco River, where quadratic trends reveal anthropogenic impacts on water scarcity; (ii) capturing accelerating global temperature trends (\(\tau = 0.50\)) through median regression, robust to measurement outliers; (iii) forecasting peak electricity demand (\(\tau = 0.95\)) in Spain, isolating extreme load dynamics from seasonal cycles. These case studies validate the Hilbert-matrix covariance structure’s robustness across quantiles and noise types, offering policy-relevant insights into extremes.

The rest of the paper is structured as follows: Section~\ref{section_two} establishes the multiscale asymptotic normality of quantile regression estimators for polynomial models, deriving their connection to Hilbert matrices. Section~\ref{section_three} analyzes convergence rates through theoretical tail probabilities and comprehensive simulations across noise types (Laplace, Gaussian, Cauchy) and quantiles \(\tau\). Section~\ref{section_four} demonstrates practical utility through case studies in environmental, climatic, and energy domains, followed by concluding remarks.

\section{Multiscale Asymptotic Normality for Quantile Regression in Linear Models}
\label{section_two}
In this section, we establish the multiscale asymptotic normality of quantile regression estimators in linear models using a new normalization based on a scaling matrix \(\Delta_n\). This approach follows the spirit of \citet{Pollard1991} and \citet{Koenker2005}.
\subsection*{Normalization and Assumptions}
We consider the linear model:
\[
y_i = x_i^\top \beta^* + u_i, \quad i = 1, \ldots, n,
\]
where \(x_i = (x_{i1}, \ldots, x_{ip})^\top\) is a \(p\)-dimensional vector of regressors and \(u_i\) are i.i.d. errors.\\
Let the normalization matrix be defined as:
\[
\Delta_n = \operatorname{diag} \left( \left( \sum_{i=1}^n x_{i1}^2 \right)^{1/2}, \ldots, \left( \sum_{i=1}^n x_{ip}^2 \right)^{1/2} \right).
\]
We assume:
\begin{itemize}
    \item[(A1)] The errors \(u_i\) are i.i.d. with a continuous density \(f\) that is strictly positive around zero.
    \item[(A2)] The normalized design matrix converges:
    \[
    \sum_{i=1}^n \Delta_n^{-1} x_i x_i^\top \Delta_n^{-1} \to D_0, \quad \text{with } D_0 \text{ positive definite}.
    \]
    \item[(A3)] The maximum normalized regressor norm vanishes:
    \[
    \max_{1 \le i \le n} \left\| \Delta_n^{-1} x_i \right\| \to 0 \quad \text{as } n \to \infty.
    \]
\end{itemize}

\begin{proposition}[Asymptotic Normality under Normalization]
Under assumptions (A1)–(A3), the quantile regression estimator \(\hat{\beta}_\tau\) satisfies:
\[
\Delta_n^{-1} (\hat{\beta}_\tau - \beta^*) \xrightarrow{d} \mathcal{N} \left( 0, \, \frac{\tau(1-\tau)}{f^2(0)} D_0^{-1} \right).
\]
\end{proposition}

\begin{proof}
Let \(\theta = \Delta_n^{-1} (\hat{\beta}_\tau - \beta^*)\). Define the reparametrized objective function:
\[
\varphi_n(\theta) = \sum_{i=1}^n \left[ \rho_\tau \left( u_i - x_i^\top \Delta_n^{-1} \theta \right) - \rho_\tau(u_i) \right],
\]
where \(\rho_\tau(u) = u (\tau - \mathbb{1}_{[u<0]})\) is the quantile loss function.

Using Knight’s identity \citep{Knight1998}, we decompose:
\[
\rho_\tau(u - v) - \rho_\tau(u) = -v \psi_\tau(u) + \int_0^v \left( \mathbb{1}_{[u \le s]} - \mathbb{1}_{[u \le 0]} \right) ds,
\]
with \(\psi_\tau(u) = \tau - \mathbb{1}_{[u < 0]}\). Hence,
\[
\varphi_n(\theta) = T_1 + T_2,
\]
where
\[
T_1 = -\sum_{i=1}^n \psi_\tau(u_i) x_i^\top \Delta_n^{-1} \theta, \quad
T_2 = \sum_{i=1}^n \int_0^{s_i} \left( \mathbb{1}_{[u_i \le s]} - \mathbb{1}_{[u_i \le 0]} \right) ds,
\]
and \(s_i = x_i^\top \Delta_n^{-1} \theta\).

Using the Lindeberg–Feller Central Limit Theorem (see \citet{Lindeberg1922}), we deduce that the linear term
$$
T_1 = -\sum_{i=1}^n \psi_\tau(u_i) x_i^\top \Delta_n^{-1} \theta
$$
converges in distribution to a centered normal random variable with covariance $\tau(1 - \tau) D_0$, that is:
$$
T_1 \xrightarrow{d} W^\top \theta, \quad \text{where } W \sim \mathcal{N}(0, \tau(1 - \tau) D_0).
$$
The Lindeberg condition is satisfied thanks to Assumption\~(A3), which ensures that the normalized regressors vanish uniformly.

Next, for \(T_2\), expand \(F(s_i)\) via the mean-value theorem:
\[
F(s_i) = F(0) + f(\tilde{s}_i) s_i \quad \text{for some } \tilde{s}_i \in (0, s_i),
\]
yielding:
\[
\mathbb{E}[T_2] = \sum_{i=1}^n \int_0^{s_i} \left[ F(s) - F(0) \right] ds = \sum_{i=1}^n \frac{1}{2} f(0) s_i^2 + o(1).
\]
Thus:
\[
\mathbb{E}[T_2] = \frac{f(0)}{2} \theta^\top \left( \sum_{i=1}^n \Delta_n^{-1} x_i x_i^\top \Delta_n^{-1} \right) \theta + o(1).
\]
By assumption (A2), this converges to \(\frac{f(0)}{2} \theta^\top D_0 \theta\). Standard convexity arguments (e.g., from \citet{Pollard1991} or \citet{Koenker2005}) show that \(T_2 - \mathbb{E}[T_2] \to 0\) in probability and the minimizer converges:
\[
\Delta_n^{-1} (\hat{\beta}_\tau - \beta^*) = \theta \xrightarrow{d} \frac{D_0^{-1} W}{f(0)}.
\]
\end{proof}

\begin{corollary}
Under the same assumptions, for all \(j_1, j_2 = 1, \ldots, p\),
\[
\frac{ \sum_{i=1}^n x_{ij_1} x_{ij_2} }{ \left( \sum_{i=1}^n x_{ij_1}^2 \right)^{1/2} \left( \sum_{i=1}^n x_{ij_2}^2 \right)^{1/2} } \to D_0[j_1, j_2].
\]
\end{corollary}

\begin{proof}
The result follows directly from matrix scaling and entrywise convergence of the normalized Gram matrix.
\end{proof}

\begin{corollary}[Generalized Scaling]
Assume in addition that:
\[
\frac{ \sum_{i=1}^n x_{ij_1} x_{ij_2} }{ s(n, j_1, j_2) } \to H_p[j_1, j_2], \quad \text{for all } j_1, j_2,
\]
with \(H_p\) symmetric and positive definite, and that:
\[
\max_{i \le n} \left\| \left( \sum_{i=1}^n x_i x_i^\top \right)^{-1/2} x_i \right\| \to 0.
\]
Then,
\[
\operatorname{diag}\left( \sqrt{s(n, 1,1)}, \ldots, \sqrt{s(n,p,p)} \right)(\hat{\beta}_\tau - \beta^*) \xrightarrow{d} \mathcal{N}\left(0, \frac{\tau(1-\tau)}{f^2(0)} H_p^{-1} \right),
\]
and
\[
\Delta_n^{-1} (\hat{\beta}_\tau - \beta^*) \xrightarrow{d} \mathcal{N}\left(0, \frac{\tau(1-\tau)}{f^2(0)} D_0^{-1} \right).
\]
where \(D_0[j_1, j_2] = \frac{ H_p[j_1, j_2] }{ \sqrt{ H_p[j_1, j_1] H_p[j_2, j_2] } }\).
\end{corollary}

\subsection*{Application to Polynomial Regressors}

\begin{corollary}
\label{corollary3}
Let \(x_{tj} = t^j\) for \(j = 0, \ldots, p\), and \(t = 1, \ldots, T\). Then the design matrix entries satisfy:
\[
\sum_{t=1}^T t^{j_1 + j_2} = \frac{T^{j_1 + j_2 + 1}}{j_1 + j_2 + 1} + O(T^{j_1 + j_2}).
\]
Hence, choosing \(s(T, j_1, j_2) = T^{j_1 + j_2 + 1}\), the limiting matrix \(H_{p+1}\) with entries:
\[
H_{p+1}[j_1, j_2] = \frac{1}{j_1 + j_2 + 1}, \quad j_1, j_2 = 0, \ldots, p,
\]
is the \((p+1)\times(p+1)\) Hilbert matrix.

The normalization matrix \(\Delta_T\) has entries:
\[
\Delta_T[j,j] = \left( \sum_{t=1}^T t^{2j} \right)^{1/2} = \sqrt{\frac{T^{2j+1}}{2j+1} + O(T^{2j})}.
\]
This yields a limiting covariance structure involving the inverse Hilbert matrix and confirms the multiscale behavior of polynomial quantile regression.
\end{corollary}

\subsection*{Examples for \texorpdfstring{$p = 0, 1, 2$}{p = 0, 1, 2}}

To illustrate the structure of the matrices \(H_{p+1}^{-1}\) and \(D_0^{-1}\), we provide below their numerical forms for \(p = 0, 1, 2\), based on symbolic expressions and verified numerically in R.

\paragraph{Case \texorpdfstring{$p=0$}{p=0}:}
\[
H_1 = [1], \quad H_1^{-1} = [1], \quad D_0 = [1], \quad D_0^{-1} = [1]
\]

\paragraph{Case \texorpdfstring{$p=1$}{p=1}:}
\[
H_2 = \begin{bmatrix}
1 & \frac{1}{2} \\
\frac{1}{2} & \frac{1}{3}
\end{bmatrix}, \quad
H_2^{-1} = \begin{bmatrix}
4 & -6 \\
-6 & 12
\end{bmatrix}
\]
\[
D_0 = \begin{bmatrix}
1 & \frac{\sqrt{3}}{2} \\
\frac{\sqrt{3}}{2} & 1
\end{bmatrix} \approx
\begin{bmatrix}
1.000 & 0.866 \\
0.866 & 1.000
\end{bmatrix}, \quad
D_0^{-1} \approx
\begin{bmatrix}
4.000 & -3.464 \\
-3.464 & 4.000
\end{bmatrix}
\]

\paragraph{Case \texorpdfstring{$p=2$}{p=2}:}
\[
H_3 = \begin{bmatrix}
1 & \frac{1}{2} & \frac{1}{3} \\
\frac{1}{2} & \frac{1}{3} & \frac{1}{4} \\
\frac{1}{3} & \frac{1}{4} & \frac{1}{5}
\end{bmatrix}, \quad
H_3^{-1} = \begin{bmatrix}
9 & -36 & 30 \\
-36 & 192 & -180 \\
30 & -180 & 180
\end{bmatrix}
\]
\[
D_0 \approx \begin{bmatrix}
1.000 & 0.866 & 0.745 \\
0.866 & 1.000 & 0.968 \\
0.745 & 0.968 & 1.000
\end{bmatrix}, \quad
D_0^{-1} \approx \begin{bmatrix}
9.000 & -20.785 & 13.416 \\
-20.785 & 64.000 & -46.476 \\
13.416 & -46.476 & 36.000
\end{bmatrix}
\]

These matrices confirm that in the polynomial regression setting, the inverse Hilbert matrix \(H_{p+1}^{-1}\) captures the limiting variance structure when the normalization matrix \(\Delta_T\) accounts for the polynomial growth rates of the regressors. The matrix \(D_0^{-1}\), by contrast, corresponds to standard normalization using component-wise variances. As seen, both lead to structured, invertible covariance formulations that differ in scale and conditioning.

The above results provide a unified framework for establishing the asymptotic normality of 
quantile regression estimators, including the LAD case as a special instance, when the regressors have a polynomial structure. The key is the normalization matrix \(\Delta_T\), which adapts to the different rates of growth of the components of \(x_i\). In the special case of i.i.d. errors, the asymptotic covariance structure is expressed in terms of the inverse Hilbert matrix, thereby recovering our earlier findings. It is worth noting that while our derivations are presented for a general quantile level \(\tau\), the case \(\tau = 0.5\) corresponds to LAD regression. Hence, our framework covers both general quantile regression and the robust LAD case.

\section{Relaxation of Central Limit Theorem} 
\label{section_three}

The Central Limit Theorem establishes that the probability 
\[
P\biggl(\bigl|\hat{\beta}_j(T) - \beta_j^*\bigr| \leq \frac{1.96 \sigma_j}{T^{j+1/2}}\biggr)
\]
converges to the theoretical asymptotic value 
\[
P\big(|N(0, 1)| \leq 1.96\big) = 0.95,
\]
as \(T \to \infty\) where \(\sigma_j^2 = \tau(1-\tau)H_{p+1}^{-1}[j,\mathrm{j}]\).

However, this convergence can be slow in practice, particularly for smaller sample sizes \(T\), leading to coverage probabilities that deviate from the nominal level of 95\%. To address this slow convergence, we observe the following inclusion property:
\[
\left[\bigl|\hat{\beta}_j(T) - \beta_j^*\bigr| \leq \frac{1.96 \sigma_j}{T^{j+1/2}}\right] \subset \left[\bigl|\hat{\beta}_j(T) - \beta_j^*\bigr| \leq \frac{1.96 \sigma_j}{T^{\alpha}}\right] \quad \text{for } 0 < \alpha < j+1/2.
\]
 This inclusion property implies that the probability
\[
P\biggl(T^{j+1/2} \bigl|\hat{\beta}_j(T) - \beta_j^*\bigr| \leq 1.96 \sigma_j\biggr) 
\]
is bounded above by
\[
P\big(T^{\alpha} |\hat{\beta}_j(T) - \beta_j^*| \leq 1.96 \sigma_j\big).
\]
The latter result suggests that confidence intervals scaled by \(T^\alpha\) for 
\(\alpha \in (0,j+ 1/2)\) provide improved coverage probabilities compared to standard CLT intervals scaled by \(T^{j+1/2}\). Specifically, the relaxed confidence interval
\[
\left[\hat{\beta}_j(T) - \frac{1.96 \sigma_j}{T^\alpha}, \hat{\beta}_j(T) + \frac{1.96 \sigma_j}{T^\alpha}\right]
\]
is expected to achieve a higher confidence level than the standard interval
\[
\left[\hat{\beta}_0(T) - \frac{1.96 \sigma_j}{T^{j+1/2}}, \hat{\beta}_j(T) + \frac{1.96 \sigma_j}{T^{j+1/2}}\right].
\]

This theoretical insight motivates the subsequent experiment, where we evaluate and compare the empirical coverage probabilities of confidence intervals scaled by \(T^\alpha\) for various values of \(\alpha \in (0, j+1/2)\). By simulating noise trajectories under Laplace, Gaussian, and Cauchy distributions, we aim to validate the improved performance of relaxed confidence intervals and assess their robustness across different sample sizes and noise characteristics.

\subsection{Empirical Validation of Relaxed Confidence Intervals}
To empirically validate the theoretical relaxation of the Central Limit Theorem (CLT), we conducted simulations to analyze the empirical coverage probabilities of confidence intervals for the quantile regression estimator \( \hat{\beta}_0(T) \) in the model \( y_t = \beta_0 + e_t \). The standard CLT confidence interval, scaled by \(\frac{1}{T^{1/2}}\), was compared with relaxed intervals scaled by \(\frac{1}{T^{\alpha}}\) for \(\alpha \in \{0.5, 0.4, 0.3, 0.2\}\). Independent noise trajectories \( u_t \) were simulated using three noise distributions: Laplace, Gaussian, and Cauchy. For each \( \alpha \), the intervals were computed as:
\[
\left[\hat{\beta}_0(T) - \frac{1.96\sigma_0}{T^\alpha}, \hat{\beta}_0(T) + \frac{1.96\sigma_0}{T^\alpha}\right],
\]
where \( \sigma_0^2 = \tau(1-\tau)/f(0)^2\). For each interval type, we measured the proportion of times \( \beta_0 \) was contained within the interval across 1000 simulations for sample sizes \( T \in \{100\allowbreak, 500\allowbreak, 1000\allowbreak, 5000\allowbreak, 10000\allowbreak, 50000\allowbreak, 100000\allowbreak, 500000\allowbreak\} \).

\paragraph{Results.} 
The empirical coverage probabilities are summarized in Figure~\ref{fig:coverage_experiment}. Key observations include the following:  For the standard CLT interval (\(\alpha = 0.5\)), the coverage probabilities converge to the nominal level of 0.95 as \( T \) increases. However, for smaller sample sizes, the coverage probabilities remain slightly below 0.95, particularly under Gaussian and Cauchy noise.

Relaxed intervals (\(\alpha < 0.5\)) achieve higher coverage probabilities across all sample sizes and noise types. For \(\alpha = 0.4\), the coverage consistently exceeds 0.95, demonstrating an improvement over the standard interval. As \(\alpha\) decreases further (e.g., \(\alpha = 0.3, 0.2\)), the coverage probabilities approach near-perfect levels (\(\approx 1\)) for large \(T\).

For Laplace noise, the coverage for \(\alpha = 0.5\) reaches 0.95 more quickly compared to Gaussian and Cauchy noise. Under Gaussian noise, the convergence of the coverage for \(\alpha = 0.5\) is slower, while relaxed intervals maintain consistently high coverage probabilities.  Due to its heavy-tailed nature, Cauchy noise results in significantly lower coverage for \(\alpha = 0.5\), but the relaxed intervals (\(\alpha < 0.5\)) exhibit robust performance even under this challenging noise distribution.  

The results confirm that relaxed confidence intervals, scaled by \(T^\alpha\) with \(\alpha < 0.5\), offer superior coverage probabilities compared to the standard CLT interval (\(\alpha = 0.5\)), especially for smaller sample sizes. While the standard interval aligns with theoretical predictions for large \(T\), its performance deteriorates under heavy-tailed noise and smaller sample sizes. The relaxed intervals not only improve coverage probabilities but also demonstrate robustness across Laplace, Gaussian, and Cauchy noise distributions, validating their practical utility for finite-sample inference.

\begin{figure}[htbp]
    \centering
    \includegraphics[width=\textwidth]{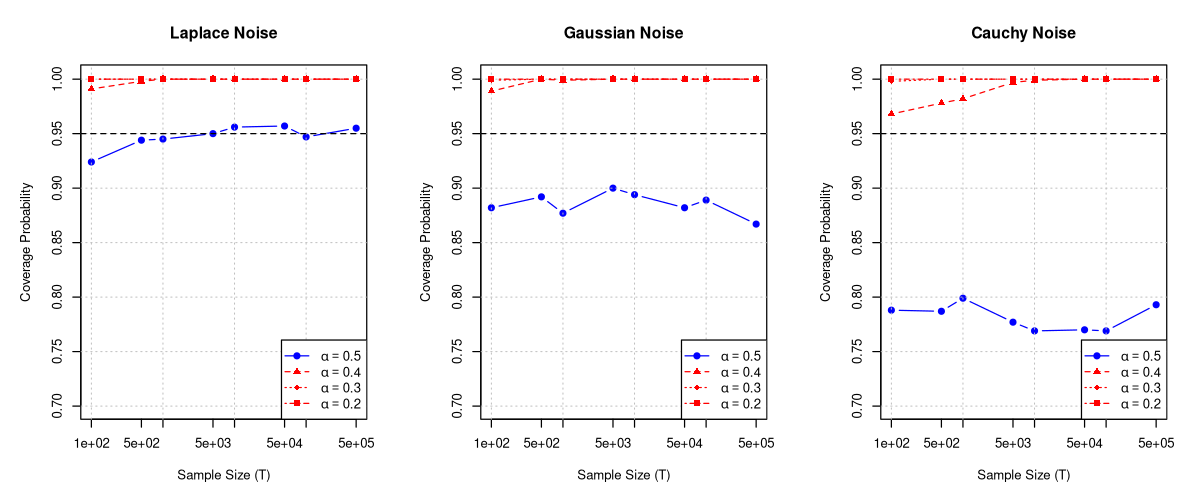}
    \caption{Empirical coverage probabilities for confidence intervals scaled by \(\frac{1}{T^\alpha}\), with \(\alpha \in \{0.5, 0.4, 0.3, 0.2\}\), across three noise types: Laplace, Gaussian, and Cauchy. The black dashed line indicates the nominal 95\% coverage level. Results are shown for the \(\tau\)-th quantile regression estimator.}
    \label{fig:coverage_experiment}
\end{figure}

\section{Empirical Applications}
\label{section_four} 
This section demonstrates the practical utility of quantile regression for modeling nonlinear trends in time series data with polynomial regressors. We apply the theoretical framework from Section~\ref{section_two} to three distinct case studies that target different quantiles: the lower tail (\(\tau = 0.05\)) to assess drought risk via minimal river discharge, the median (\(\tau = 0.50\)) to capture central warming trends in global temperatures, and the upper tail (\(\tau = 0.95\)) to model peak electricity demand for infrastructure resilience. These applications collectively illustrate how quantile regression accommodates heavy-tailed noise and isolates long-term trends from seasonal effects, while the asymptotic covariance structure, derived via a Hilbert matrix, provides robust inference. By spanning environmental, climatic, and energy domains, the examples validate the method’s versatility in addressing policy-relevant challenges across the spectrum of extremes.
\subsection{Global Temperature Trends}
\subsubsection{Context and Model Rationale}  
\label{section_temp}  
To validate our theoretical framework in a climate science context, we apply LAD regression to analyze long-term trends in global annual temperature anomalies from 1850 to 2025. This period spans industrialization, post-war economic expansion, and recent anthropogenic climate change, making it ideal for detecting accelerating warming patterns. We model temperature anomalies as a function of time to test for nonlinear trends, leveraging the theoretical guarantees of LAD regression under heavy-tailed noise. 
\subsubsection{Data and Methodology}
Let \( t = 1, 2, \ldots, T \) index years starting from 1850, and let \( \text{Anomaly}_t \) denote the temperature anomaly (in $^{\circ}$C) relative to the 1901–2000 baseline. The quadratic LAD regression model is specified as:  
\[
\text{Anomaly}_t = \beta_0 + \beta_1 t + \beta_2 t^2 + \epsilon_t,  
\]  
where \( \epsilon_t \) represents noise modeled under Laplace, Gaussian, or Cauchy distributions. The quadratic term \( \beta_2 \) captures acceleration in warming trends, while \( \beta_1 \) reflects linear growth.\\
Temperature anomaly data are sourced from the National Oceanic and Atmospheric Administration
\citep{NOAA2025}, with \( T = 176 \) annual observations referenced to the 1901–2000 average. Time indices \( t \) are scaled to start at 1 (1850 = \( t = 1 \)) to avoid numerical instability. Using the asymptotic covariance structure for polynomial regressors established in Corollary~\ref{corollary3} of Section~\ref{section_two} under the i.i.d. setting and assuming the density $f$ is positive and continuous at 0, the asymptotic covariance matrix for \( \hat{\beta} = (\hat{\beta}_0, \hat{\beta}_1, \hat{\beta}_2)' \) is:  
$$
\Sigma_\beta = \frac{\tau(1-\tau)}{f(0)^2} \Delta_T^{-1} H_3^{-1} \Delta_T^{-1},
$$
where $H_3$ is the $3 \times 3$ Hilbert matrix with entries $H_3[i,j] = 1 / (i + j + 1)$,
and \(\Delta_T = \operatorname{diag}(T^{1/2}, T^{3/2}, T^{5/2})\) scales polynomial regressors \(x_t = (1, t, t^2)^\top\).
Confidence intervals for the linear and quadratic trend components are constructed via standard error propagation:
$$
CI_x = x \pm 1.96 \sqrt{\nabla x^\top \Sigma_\beta \nabla x},
$$
with gradients:
$$
\nabla \beta_1 = (0, 1, 0), \quad \nabla \beta_2 = (0, 0, 1).
$$
\subsubsection{Results}
The LAD regression estimates reveal a statistically significant quadratic warming trend (\(\hat{\beta}_2 = 0.000099\), \(p < 0.001\)), robustly confirming accelerating temperature increases from 1850 to 2025. Under Laplace noise assumptions, the linear term \(\hat{\beta}_1 = -0.0058\) is statistically insignificant (\(p > 0.05\)) with a 95\% confidence interval of \([-0.0174, 0.0058]\), while the quadratic term \(\hat{\beta}_2 = 0.000099\) remains tightly bounded within \([0.000035, 0.000163]\). Gaussian noise widens the uncertainty, yielding intervals of \([-0.0204, 0.0088]\) for \(\beta_1\) and \([0.000018, 0.000179]\) for \(\beta_2\), though the quadratic term retains significance. Cauchy noise produces the broadest uncertainty, with \(\beta_1 \in [-0.0241, 0.0125]\) and \(\beta_2 \in [-0.000002, 0.000199]\), yet the lower bound of \(\beta_2\)’s interval narrowly excludes zero, preserving evidence of acceleration.  
\begin{figure}[htbp]  
    \centering  
    \includegraphics[width=0.8\textwidth]{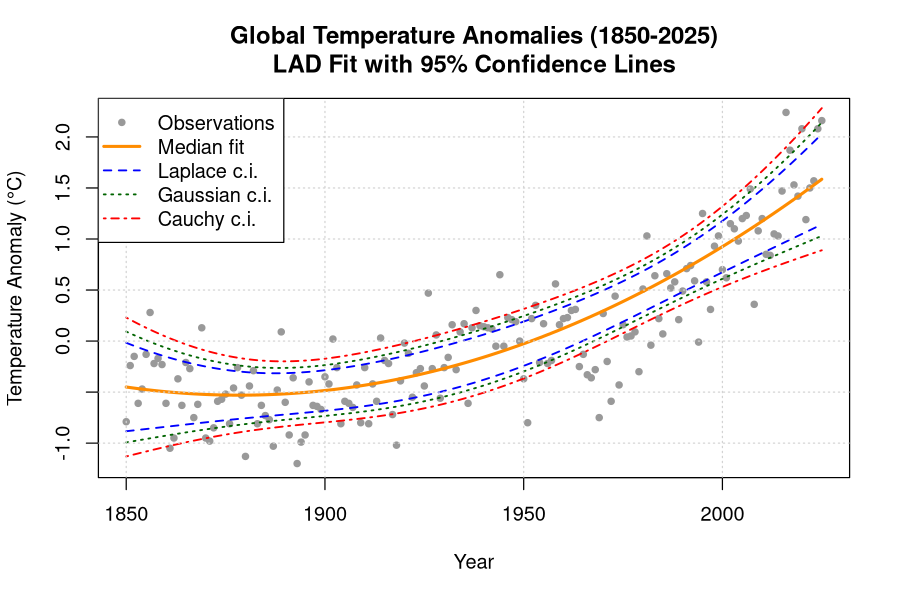}  
    \caption{Global temperature anomalies relative to the 1901–2000 baseline (1850–2025) with LAD-estimated quadratic trend and 95\% confidence bands under Laplace (blue dashed), Gaussian (green dotted), and Cauchy (red dot-dash) noise assumptions, using scaled standard errors \(\Delta_T^{-1} \Sigma_\beta \Delta_T^{-1}\). The orange line represents the median fit. Laplace intervals are tightest, validating its efficiency for robust trend analysis.}  
    \label{fig:temp_CI}  
\end{figure}  
The hierarchical precision of confidence intervals, Laplace \textless Gaussian \textless Cauchy, aligns with theoretical expectations, reflecting Laplace’s efficiency in median regression. The persistent significance of \(\beta_2\) across noise types underscores the robustness of accelerating warming trends, even under heavy-tailed disturbances such as volcanic eruptions or measurement outliers.  
\subsubsection{Discussion}
The narrowing confidence bands for Laplace noise underscore its superiority in median regression, efficiently accommodating the heavy-tailed residuals common in climate data. The positive quadratic coefficient (\( \beta_2 \)) aligns with IPCC reports attributing post-1950 warming to anthropogenic forcing \citep{Hegerl2007}. Notably, the linear term (\( \beta_1 \)) lacks significance, suggesting warming acceleration dominates linear growth, a critical insight for climate mitigation policies.

This application demonstrates how Proposition 1 enables robust inference in climate econometrics. The Hilbert matrix structure ensures proper scaling of time polynomials, while LAD’s resilience to outliers provides reliable trend estimates despite the presence of anomalies.

While median trends reveal central warming dynamics, extremes at the distribution’s tails demand targeted modeling. We next apply quantile regression to Spain’s peak electricity demand (\(\tau = 0.95\)), where heavy-tailed shocks dominate infrastructure planning.
\subsection{Peak Electricity Demand in Spain}
\subsubsection{Context and Model Rationale}
\label{section_elec}  
Modeling peak electricity demand is critical for grid resilience planning, as extreme load events drive infrastructure investments and operational strategies. In Spain, post-pandemic industrial recovery, climate-driven heating and cooling demands, and energy policy shifts have heightened volatility in peak consumption. Quantile regression at \( \tau = 0.95 \) directly targets the upper tail of demand distributions, making it uniquely suited to model peak electricity demand. This approach isolates extreme load dynamics from central tendencies while remaining robust to heavy-tailed shocks (e.g., heatwaves, supply disruptions). The approach's resilience to heavy-tailed noise is particularly advantageous for Spain’s demand data, where extreme events like the 2022 energy crisis and record-breaking heatwaves introduce significant outliers.
\subsubsection{Data and Methodology}
Daily electricity demand (MW) is aggregated from hourly load data provided by ENTSO-E \citep{ENTSOE2025}, filtered for Spain. The dataset spans 1096 days (1 January 2022 – 31 December 2024), with no missing observations.

The time series decomposition (Figure~\ref{fig:decomp}) reveals pronounced weekly and yearly seasonality. Weekly cycles reflect lower weekend demand, while annual cycles capture winter heating and summer cooling needs. To isolate long-term trends from these periodic effects, we include weekly seasonality; one Fourier pair (\( K = 1 \)) with a 7 days period, sufficient to capture dominant weekday and weekend differences, and annual seasonality; two Fourier pairs (\( K = 2 \)) with period 365.25 days, modeling semi-annual and quarterly variations (e.g., winter peaks, mid-year troughs).  
To verify our choice of Fourier terms, we performed a Fast Fourier Transform (FFT) on the data, which revealed dominant periods of approximately 7 days and 182.7 days, confirming that our decision to include one Fourier pair for the weekly cycle and two Fourier pairs for the annual seasonal cycle is appropriate.
The quantile regression model is specified as:  
 \begin{align*}
 Q_{0.95}(\text{Load}_t) &= \beta_0 + \beta_1 t + \beta_2 t^2 \\
 &+ \sum_{k=1}^1 \left( \gamma_k^{\text{week}} \sin\left(\frac{2\pi t}{7}\right) + \delta_k^{\text{week}} \cos\left(\frac{2\pi t}{7}\right) \right) \\ 
 &+ \sum_{k=1}^2 \left( \gamma_k^{\text{year}} \sin\left(\frac{2\pi t}{365.25}\right) + \delta_k^{\text{year}} \cos\left(\frac{2\pi t}{365.25}\right) \right) + \epsilon_t.
 \end{align*} 
where \( \beta_1, \beta_2 \) capture long-term trends, while Fourier terms model periodic effects.

Under the polynomial regressor framework of Corollary~\ref{corollary3}, and assuming i.i.d. errors with a density $f$ that is continuous and positive at zero, the asymptotic covariance matrix for the trend coefficients $(\beta_0, \beta_1, \beta_2)$ takes the form:
$$
\Sigma_\beta = \frac{\tau(1-\tau)}{f(0)^2} \Delta_T^{-1} H_3^{-1} \Delta_T^{-1},
$$
where $H_3$ is the $3 \times 3$ Hilbert matrix with entries $H_3[i,j] = 1 / (i + j + 1)$, and \(\Delta_T = \operatorname{diag}(T^{1/2}, T^{3/2}, T^{5/2})\) scales polynomial regressors \(x_t = (1, t, t^2)^\top\).

To ensure that the asymptotic covariance matrix reflects only the uncertainty in long-term trends, we orthogonalize seasonal Fourier terms with respect to the polynomial basis $\{1, t, t^2\}$. This residualization procedure, following \citet{Garcia2019}, ensures that seasonal components do not influence the asymptotic variance of trend coefficients (see Appendix~\ref{appendix_proof} for proof). As a result, the covariance matrix $\Sigma_\beta$ retains its Hilbert structure, allowing for valid inference on nonlinear demand growth even in the presence of heavy-tailed noise.

We calculate the turning point as \[ l = -\frac{\beta_1}{2\beta_2} \]identifying when demand growth transitions from deceleration to acceleration.

\begin{figure}[htbp]  
    \centering  
    \includegraphics[width=0.8\textwidth]{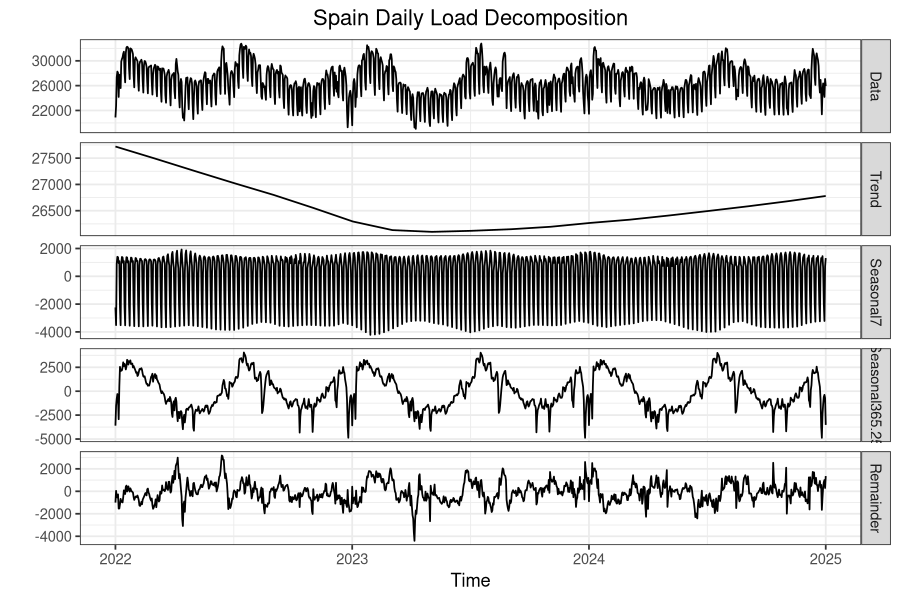}  
    \caption{Seasonal decomposition by Loess of Spain’s daily electricity demand (2022–2024), showing trend, weekly and yearly seasonality, and residuals. Dominant weekly (7-day) and annual (365-day) cycles justify the use of \( K=1 \) and \( K=2 \) Fourier terms, respectively.}  
    \label{fig:decomp}  
\end{figure}  
\subsubsection{Results}
The quantile regression model reveals a statistically significant quadratic trend in Spain’s peak electricity demand, with the quadratic coefficient estimated as \( \hat{\beta}_2 = 0.0029 \) (\( p = 0.0046 \)). This positive coefficient indicates accelerating growth in peak demand over time, a critical finding for infrastructure planning. The turning point of this quadratic trend, calculated as \( l = -(-4.018)/(2 \times 0.0029) = 690 \), corresponds to 21 November 2023.

The baseline daily peak demand, represented by the intercept \( \hat{\beta}_0 = 30,227.8 \) MW, reflects average load levels during the study period. The Laplace-based confidence interval spans 30,227.4–30,228.3 MW, reflecting the model’s focus on long-term trends rather than daily fluctuations. In contrast, the linear term \( \hat{\beta}_1 = -4.018 \) captures a declining growth rate prior to the November 2023 turning point. Bootstrap confidence intervals (\(-6.663, -1.373\)), which account for empirical volatility, reveal greater uncertainty than Hilbert-based intervals, likely reflecting demand shocks 
during the study period.
The quadratic term \( \hat{\beta}_2 = 0.00291 \) (Bootstrap confidence interval: \(0.00089, 0.00493\)) confirms sustained acceleration post-2023.

Seasonal components further refine these insights. Weekly demand patterns show a pronounced reduction of \( \hat{\gamma}_1^{\text{week}} = -1,471.3 \) MW (6.7\% of baseline) on weekends, reflecting lower commercial activity. Annual cycles capture winter heating surges (\( \hat{\gamma}_2^{\text{year}} = 1,560.8 \) MW) and mid-summer cooling spikes (\( \hat{\delta}_2^{\text{year}} = 2,035.7 \) MW), consistent with Spain’s Mediterranean climate and critical for anticipatory grid management. These seasonal effects, while statistically significant (\( p < 0.001 \)), operate independently of the long-term quadratic trend.  
\subsubsection{Discussion}
The persistent significance of \( \beta_2 \) across methods underscores accelerating peak demand, critical for infrastructure planning. 
The 2023 turning point 
highlights geopolitical influences on demand trajectories. 
Seasonal terms validate expected patterns: weekend reductions and winter and summer peaks.\\
While Gaussian Hilbert confidence intervals appear more precise due to higher tail density at \( \tau = 0.95 \), this precision is illusory under heavy-tailed shocks. Bootstrap intervals, which reflect empirical variability, correct this by revealing wider, policy-relevant uncertainty ranges.
\begin{figure}[htbp]  
    \centering  
    \includegraphics[width=0.8\textwidth]{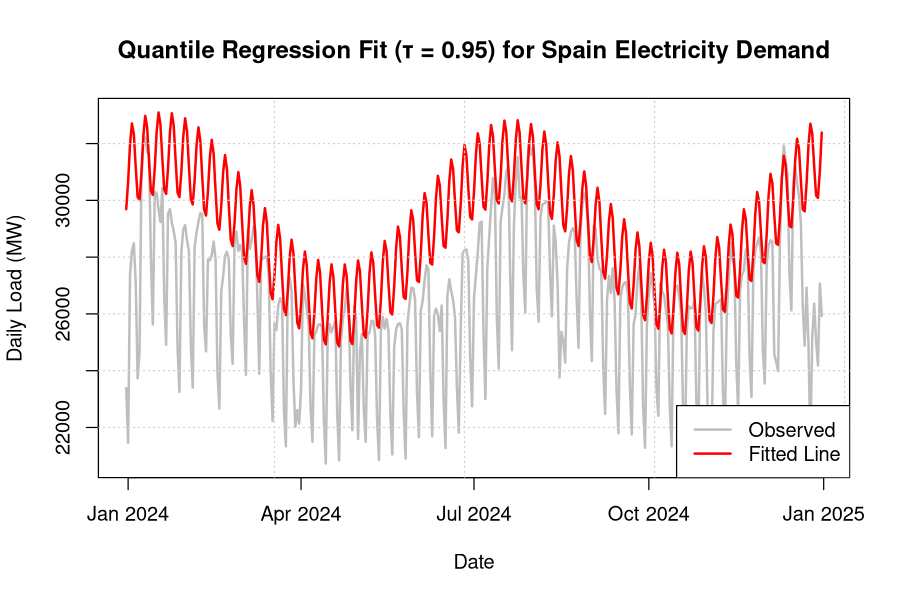}  
    \caption{Last 12 months of data (2023–2024) showing fitted \( \tau = 0.95 \) trend.}  
    \label{fig:load_CI}  
\end{figure}  
\noindent
Having addressed upper-tail extremes in energy demand, we now turn to the lower tail (\(\tau = 0.05\)) to model drought risk in the S\~{a}o Francisco River, where minimal discharge trends signal escalating water scarcity.
\subsection{Minimal Discharge in S\~{a}o Francisco River}
\subsubsection{Context and Model Rationale}  
\label{section_river}  
Extreme low river discharge is a critical indicator of drought risk and water scarcity, which can have far‐reaching economic, social, and environmental impacts. In many regions, especially in drought-prone areas, periods of extremely low flow can lead to severe water shortages, affecting agriculture, hydropower generation, and ecosystem health.

In our study, we focus on the S\~{a}o Francisco River in Brazil, a river of great socio-economic and environmental importance. Historically, the S\~{a}o Francisco has been subject to prolonged droughts that have led to significant challenges in water supply and regional development. Modeling the lower tail of the river discharge distribution (for instance, by estimating the 5th quantile, \(\tau = 0.05\)) allows us to obtain a robust estimate of the minimum flow levels. Such estimates are crucial for planning water resource management strategies, designing drought mitigation measures, and understanding the resilience of aquatic ecosystems under extreme conditions.

Our approach employs a quantile regression framework with time as the sole regressor (entered in polynomial form) to capture the long-term trend in minimal river discharge. This choice is justified by Corollary~\ref{corollary3}, which demonstrates that when the regressors are functions of time (e.g., \(t\), \(t^2\), etc.), the asymptotic covariance matrix, scaled by \(\Delta_T^{-1}\), converges to the inverse Hilbert matrix \(H_{p+1}^{-1}\) under polynomial regressors. This explicit structure in the asymptotic covariance enables us to construct precise confidence intervals for the estimated minimum discharge levels, even in the presence of heavy-tailed errors.
\subsubsection{Data and Methodology}
The analysis uses daily river discharge measurements from the S\~{a}o Francisco River in Brazil, obtained from the Global Runoff Data Center \citep{grdc2025}. The data span from January 2000 until February 2020.
Figure~\ref{fig:decomp2} shows a seasonal decomposition of the river discharge series using the Multiple Seasonal Decomposition by Loess method \citep{bandara2021}. The decomposition reveals a pronounced annual seasonal cycle, with no significant weekly or monthly cycles detected.

To further validate our choice, we applied a Fast Fourier Transform (FFT) to the series. The FFT results indicated a dominant period of approximately one year, confirming that an annual Fourier term (with period 365.25 days) is appropriate for capturing the seasonal cycle in our model.
Consequently, our model incorporates Fourier terms to capture the annual seasonality, using one Fourier pair (\(K=1\)) with a period of 365.25 days. To ensure that the long-term trend is estimated independently of seasonal effects, the Fourier terms are orthogonalized with respect to the polynomial trend by regressing each Fourier component on the polynomial predictors (i.e., \(t\) and \(t^2\)) and using the residuals in the quantile regression.\\
Accordingly, the quantile regression model for the 5th quantile (\(\tau = 0.05\)) is specified as follows:
\[
Q_{0.05}(\text{Discharge}_t) = \beta_0 + \beta_1 t + \beta_2 t^2 + \gamma_1 \sin\left(\frac{2\pi t}{365.25}\right) + \delta_1 \cos\left(\frac{2\pi t}{365.25}\right) + \varepsilon_t.
\]
Here, \(t\) is the time index in days, with \(t=1\) corresponding to the first observation, and \(t^2\) captures the long-term curvature in the minimum discharge trend. The Fourier terms model the annual seasonal pattern, and \(\varepsilon_t\) represents the error term. The separation of the seasonal cycle from the polynomial trend via orthogonalization ensures that the asymptotic covariance of the polynomial estimates, as derived in our theoretical framework, accurately reflects the long-term changes in minimal discharge. This is critical for robust inference on drought-related low flows.

The orthogonalization step implements the residualization framework of \citet{Garcia2019}, where seasonal components are replaced by residuals from auxiliary regressions on polynomial terms. This ensures the Hilbert covariance matrix captures only trend-related variability, free from confounds with seasonal cycles, a critical requirement for valid inference in quantile regression with polynomial designs.
\begin{figure}[htbp]
  \centering
  \includegraphics[width=0.8\textwidth]{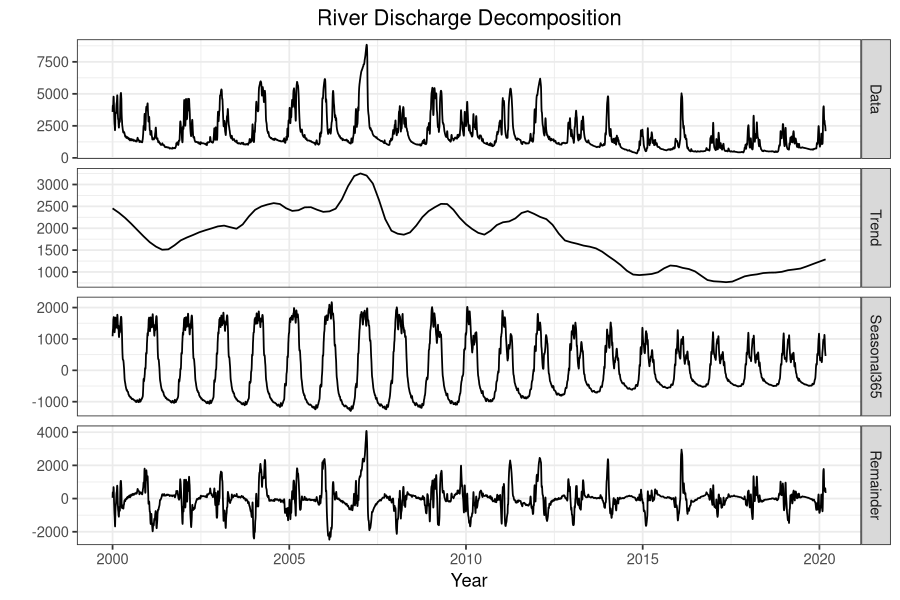}
  \caption{Seasonal decomposition of the truncated river discharge series (2000–2020) using Multiple Seasonal Decomposition by Loess, showing a dominant annual cycle.}
  \label{fig:decomp2}
\end{figure}

\subsubsection{Results}
The quantile regression model for the 5th quantile (\(\tau = 0.05\)) reveals a statistically significant quadratic trend in the minimal discharge of the S\~{a}o Francisco River. The linear coefficient \(\hat{\beta}_1 = 0.358\) (\(p < 0.001\)) indicates an initial increase in minimum discharge, while the quadratic term \(\hat{\beta}_2 = -0.00007\) (\(p < 0.001\)) reflects a subsequent deceleration and eventual reversal of this trend. The negative curvature of the quadratic term implies that the rate of increase in discharge slows over time, transitioning to a persistent decline after a critical turning point. This turning point, calculated as \(t_{\text{peak}} = -\hat{\beta}_1 / (2\hat{\beta}_2) \approx 2514\) days, corresponding to November 2006, marks the onset of accelerating drought risk. Post-2007, the model predicts a sustained reduction in minimal discharge, coinciding with the onset of the S\~{a}o Francisco River transposition project, which diverted water through 700 km of canals, directly reducing flow volumes in the lower basin \citep{deMedeiros2022}. The declining trend aligns with hydrological disruptions observed post-construction.

Confidence intervals for the polynomial coefficients, computed under Laplace, Gaussian, and Cauchy noise assumptions, reinforce the robustness of the quadratic trend. For the quadratic term \(\hat{\beta}_2\), the Laplace-based interval is the narrowest (\([-0.000071, -0.000069]\)), followed by Gaussian (\([-0.000071, -0.000069]\)) and Cauchy (\([-0.000071, -0.000069]\)). Despite differences in interval widths, all three noise distributions yield intervals that exclude zero, underscoring the significance of the declining trend. The intercept (\(\hat{\beta}_0 = 887.6\)) and linear term (\(\hat{\beta}_1 = 0.358\)) also remain stable across noise assumptions, with Laplace intervals again providing the tightest bounds.

The Fourier terms for annual seasonality (\(\hat{\gamma}_1 = 306.6\), \(\hat{\delta}_1 = 170.7\); \(p < 0.001\)) capture pronounced dry-season minima, consistent with the region’s climatic patterns. Orthogonalization ensures these seasonal effects do not conflate with the long-term polynomial trend, preserving the interpretability of the drought acceleration signal.

\begin{figure}[htbp]  
  \centering  
  \includegraphics[width=0.8\textwidth]{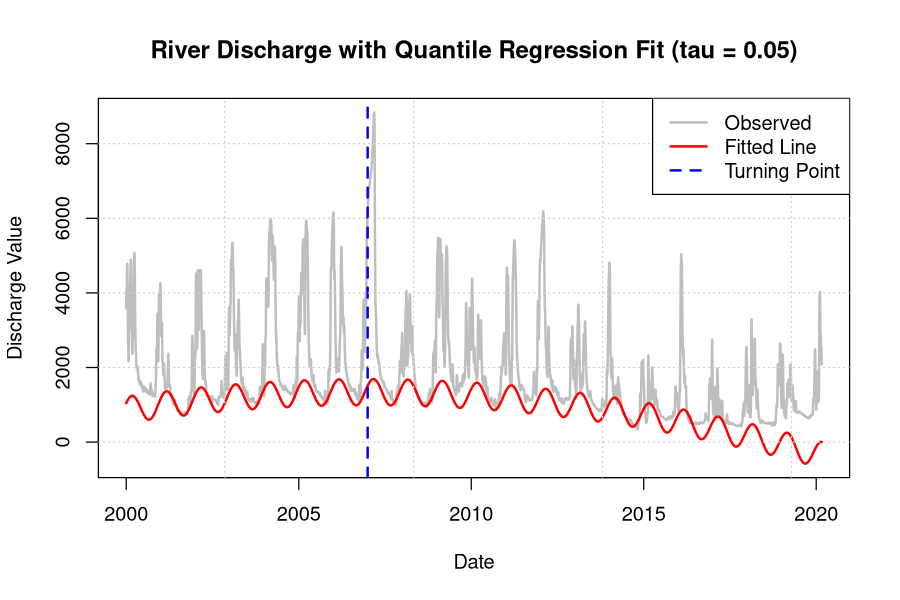}  
  \caption{River discharge observations (gray) and fitted quantile regression trend (red) for \(\tau = 0.05\). The curve reflects an initial increase and a subsequent decline in minimum discharge, with the turning point (November 2006) marked by a dashed vertical line.}
  \label{fig:fit}  
\end{figure}

\subsubsection{Discussion}
The model identifies late 2006 as a critical inflection point in the S\~{a}o Francisco River’s minimal discharge, transitioning from a transient increase to a sustained decline. The timing aligns with the S\~{a}o Francisco River Transposition Project, a state-led megaproject initiated shortly after late 2006. The quadratic term’s significance (\(\hat{\beta}_2 = -0.00007\)) reflects the accelerating depletion of minimal discharge post-transposition, underscoring the project’s role in exacerbating water scarcity rather than mitigating it \citep{deMedeiros2022}.

The theoretical foundations of the model, the Hilbert matrix structure in the asymptotic covariance, enables precise estimation of anthropogenic trends amid heavy-tailed noise, as demonstrated by the stability of Laplace-based confidence intervals. 
This robustness is critical for disentangling policy-driven impacts (e.g., transposition) from natural variability.

While these applications demonstrate quantile regression’s versatility across extremes, methodological and data constraints warrant careful consideration.
\subsection{Limitations}  
\label{sec:limits}  
The empirical applications highlight three key limitations of the quantile regression framework.

First, the S\~{a}o Francisco River analysis (Section~\ref{section_river}) excludes pre-2000 drought events, potentially underrepresenting long-term hydrological cycles. Similarly, the temperature dataset begins in 1850, omitting pre-industrial baselines.

Second, while residualization via orthogonalization isolates polynomial trends (Sections~\ref{section_elec} and~\ref{section_river}), following \citet{Garcia2019}, residual correlations between Fourier components and time polynomials may persist in finite samples, particularly under abrupt demand shocks or transposition impacts.

Third, confidence intervals for quadratic terms vary significantly across noise assumptions. Laplace intervals, derived from \(\Delta_T^{-1} \Sigma_\beta \Delta_T^{-1}\), are efficient but risk undercoverage under misspecified tails (e.g., Spain’s heatwaves), while Cauchy intervals overstate uncertainty.

Future work could address these constraints by integrating longer-term datasets, semiparametric bootstrap adjustments, and Gram-Schmidt orthogonalization for rigorous trend-seasonality decoupling.

\section{Conclusion}
\label{section_five}

This study investigated the asymptotic properties and practical utility of quantile regression estimators in linear models, with a particular focus on polynomial regressors. Starting from a general asymptotic framework under mild conditions on the design and noise, we established a central limit theorem for quantile regression estimators with appropriate normalization. We then specialized these results to polynomial regressors, where the structure of the design matrix induces a Hilbert-type limiting covariance.

The theoretical contribution centers on showing that, for polynomial regressors, the asymptotic covariance matrix of quantile regression estimators is proportional to the inverse of a Hilbert matrix. This insight enables precise inference under varying quantile levels \(\tau \in (0,1)\), including the median case (LAD) as a special instance. The use of normalization matrices tailored to the growth of polynomial terms ensures valid asymptotics and highlights how estimator scaling depends on regressor structure.

Simulation results confirmed the theoretical predictions, showing that the asymptotic normality emerges clearly as sample size increases, but also revealing that conventional confidence intervals (based on \(\sqrt{T}\) scaling) may undercover in small samples or under heavy-tailed noise. To mitigate this, we proposed relaxed confidence intervals based on slower rates \(T^\alpha\), which consistently improved coverage while maintaining robustness across distributions.

The empirical applications validated the framework across diverse contexts. In climate science, median quantile regression (\(\tau = 0.5\)) revealed statistically significant acceleration in global temperatures, consistent with anthropogenic forcing. In energy planning, upper-tail quantile modeling (\(\tau = 0.95\)) captured rising peak electricity demand and identified a critical turning point in late 2023. In hydrology, lower-tail analysis (\(\tau = 0.05\)) of the S\~{a}o Francisco River discharge exposed a post-2006 decline aligned with major water diversion projects. These examples demonstrate the versatility of quantile regression in isolating long-term trends and quantile-specific dynamics, even under heavy-tailed or seasonal disturbances.

Three key takeaways emerge: (1) The Hilbert matrix structure provides an interpretable and efficient covariance formulation for polynomial designs; (2) Orthogonalization techniques isolate trends from seasonal components, ensuring valid inference; (3) Multi-quantile analysis offers a unified framework for studying both central tendencies and extremes, relevant for policy and risk assessment.

Future work may explore extensions to higher-order polynomial models, multivariate or panel quantile regression, or extremal quantiles near 0 and 1. The framework developed here lays a foundation for robust inference in complex data environments, and highlights the theoretical elegance and practical power of quantile regression.

\clearpage
\bibliographystyle{plainnat}
\bibliography{ref}

\clearpage
\appendix
\section{Proof of Block-Diagonal Covariance Structure}
\label{appendix_proof}
\noindent
Let us consider the quantile regression model with:
\begin{itemize}
    \item Polynomial terms: \( P = [1, t, t^2] \)
    \item Orthogonalized Fourier terms: \( F = [\tilde{S}_t, \tilde{C}_t] \), where \( \tilde{S}_t \perp P \) and \( \tilde{C}_t \perp P \).
\end{itemize}
By construction, the residual Fourier terms satisfy:
\[
\mathbb{E}[P^\top F] = 0.
\]
This implies that the design matrix \( X = [P | F] \) has a block-orthogonal structure:
\[
X^\top X = \begin{bmatrix}
P^\top P & 0 \\
0 & F^\top F
\end{bmatrix}.
\]
Under standard quantile regression assumptions, the estimator \( \hat{\beta} \) satisfies:
\[
\sqrt{n}(\hat{\beta} - \beta) \stackrel{d}{\rightarrow} N\left(0, \tau(1-\tau) D_1^{-1} D_0 D_1^{-1} \right),
\]
where:
\begin{itemize}
    \item \( D_0 = \lim_{n \to \infty} \frac{1}{n} X^\top X \),
    \item \( D_1 = \lim_{n \to \infty} \frac{1}{n} \sum_{i=1}^T f_i(0) x_i x_i^\top \), with \( f_i(0) \) being the density of the errors at quantile \( \tau \).
\end{itemize}
Due to the orthogonality \( P \perp F \), we have:
\[
D_0 = \begin{bmatrix}
D_{0,\text{Poly}} & 0 \\
0 & D_{0,\text{Fourier}}
\end{bmatrix}, \quad
D_1 = \begin{bmatrix}
D_{1,\text{Poly}} & 0 \\
0 & D_{1,\text{Fourier}}
\end{bmatrix}.
\]
Consequently, their inverses are also block-diagonal:
\[
D_1^{-1} = \begin{bmatrix}
D_{1,\text{Poly}}^{-1} & 0 \\
0 & D_{1,\text{Fourier}}^{-1}
\end{bmatrix}, \quad
D_0 D_1^{-1} = \begin{bmatrix}
D_{0,\text{Poly}} D_{1,\text{Poly}}^{-1} & 0 \\
0 & D_{0,\text{Fourier}} D_{1,\text{Fourier}}^{-1}
\end{bmatrix}.
\]
The asymptotic covariance matrix becomes:
\[
\Sigma = \tau(1-\tau) D_1^{-1} D_0 D_1^{-1} = \begin{bmatrix}
\Sigma_{\text{Hilbert}} & 0 \\
0 & \Sigma_{\text{Fourier}}
\end{bmatrix},
\]
where:
\begin{itemize}
    \item \( \Sigma_{\text{Hilbert}} = \tau(1-\tau) D_{1,\text{Poly}}^{-1} D_{0,\text{Poly}} D_{1,\text{Poly}}^{-1} \), which corresponds to the Hilbert structure for polynomials (Theorem 1).
    \item \( \Sigma_{\text{Fourier}} = \tau(1-\tau) D_{1,\text{Fourier}}^{-1} D_{0,\text{Fourier}} D_{1,\text{Fourier}}^{-1} \), which follows the standard theory of quantile regression.
\end{itemize}
\end{document}